\documentclass{proc-l}
\usepackage{amsrefs}

\newtheorem{theorem}{Theorem}[section]
\newtheorem{lemma}[theorem]{Lemma}

\theoremstyle{definition}
\newtheorem{definition}[theorem]{Definition}
\newtheorem{example}[theorem]{Example}
\newtheorem{remark}[theorem]{Remark}
\newtheorem{corollary}[theorem]{Corollary}
\newtheorem{notation}[theorem]{Notation}
\newtheorem{proposition}[theorem]{Proposition}

\numberwithin{equation}{section}

\usepackage{amsthm}
\usepackage{latexsym}
\usepackage{color}
\usepackage{tikz}
\usepackage{pgfplots}
\usepackage{pgfplotstable}
\usepackage[utf8]{inputenc}

\pgfplotsset{compat = newest}

\begin{document}

\title{On the scale-freeness of random colored substitution networks}


\author{Nero Ziyu Li}
\address{Nero Ziyu Li: Department of Mathematics, Imperial College London, South Kensington Campus, London SW7 2AZ, United Kingdom.}
\email{z5222549@zmail.unsw.edu.au, ziyu.li21@imperial.ac.uk}
\thanks{The first author has been supported by
the Additional Funding Programme for Mathematical Sciences, delivered by EPSRC (EP/V521917/1) and the Heilbronn Institute for Mathematical Research, and also by the EPSRC Centre for Doctoral Training in Mathematics of Random Systems: Analysis, Modelling and Simulation (EP/S023925/1).}

\author{Thomas Britz}
\address{Thomas Britz: School of Mathematics and Statistics, UNSW Sydney, Sydney NSW 2052, Australia.}
\email{britz@unsw.edu.au}


\date{}

\dedicatory{}

\commby{Zhen-Qing Chen}

\begin{abstract}
Extending previous results in the literature, 
random colored substitution networks and degree dimension are defined in this paper.
The scale-freeness of these networks is proved by introducing 
a new definition for degree dimension that is associated with Lyapunov exponents.
The random colored substitution network hence turns out to be a simple, 
powerful and promising model to generate random scale-free networks.
\end{abstract}

\maketitle

Many real-life phenomena are fractals in nature, 
including growing networks found in biology, brain connections and in social interactions.
Previous researchers introduced a mathematical model called {\em substitution networks},
to simulate the growth of the networks by iteratively replacing each arc of a network by smaller networks.
This model was later expanded by the introduction of arc colors to allow more types of arc replacements. 
However, these models are deterministic and do not allow for the randomness that real-life growth networks can exhibit. 
To capture this randomness, 
we expand the model to what we call {\em random colored substitution networks}, 
by allowing each arc to be replaced by a random choice of network.
We describe the properties of the randomly resulting networks, 
including their number of nodes and arcs and their node degrees.
Our main result shows that these random colored substitution networks are almost surely {\em scale-free} 
and that they therefore have a particular type of structure.

\section{Introduction}
\label{sec:introduction}

The property of {\em scale-freeness} of complex networks was first proposed in 1999 by Albert-Laszlo Barabasi and Reka Albert~\cite{BaAl99}. 
Their model presents a graph that grows by the addition of new nodes and their incident edges to existing nodes.
The probability that a new node is chosen to be adjacent to an old node depends in this model on the degree of the old node.
This model is shown to be {\em scale-free}, 
which refers to the phenomenon of networks having node degrees that obey the power-law distribute on.
In particular, the fraction $P(k)$ of nodes in the network $G$ that are adjacent to~$k$ other nodes is
\begin{equation}
\label{eq:original scale-free}
  P(k)\sim k^{-\delta}
\end{equation}
where $\sim$ denotes approximation and $\delta$ is some constant. 
Scale-free networks have appeared, among other places, in studies on 
         biology~\cites{Biology5,Biology4,Biology1,Biology2,Biology3}, 
         finance~\cites{Finance3,Finance2,Finance1} and 
computer science~\cites{Computer1,Computer2}.

Barabasi and Albert's model is suitable for modelling many real-world complex networks.
However, for modelling growth processes that evolve in fractal-like ways, 
a more suitable type of complex network model is the {\em substitution network}.
These networks were first introduced by Xi et al.~\cite{XiWaWaYuWa17} 
and are models featuring networks whose arcs are at each step replaced by some fixed network.
An example of the first four steps of such a substitution network is given in Fig.~\ref{fig:substitution}.

\definecolor{myaqua}{RGB}{15,205,255}

\newcommand{\redarc}[2]{\draw[draw=red,             thick,>=stealth,->] (#1) -- (#2);\fill[black] (#1) circle (1.5pt);\fill[black] (#2) circle (1.5pt);}
\newcommand{\bluearc}[2]{\draw[draw=myaqua!90!black,thick,>=stealth,->] (#1) -- (#2);\fill[black] (#1) circle (1.5pt);\fill[black] (#2) circle (1.5pt);}

\newcommand{\mynode}[2]{\begin{scope}[shift={(#1)},scale=#2]\node[node] (0,0) () {};\end{scope}}
\newcommand{\Rrule}[4]{\begin{scope}[shift={(#1)},scale=#2,rotate=#3]
  \foreach \nn/\x/\y in {0/0/0, 1/1/0, 2/1/1, 3/2/0, 4/2/1, 5/3/0}{\node[node] (\nn) at (\x,\y) {};}
  \draw[black] (0) -- (5) (1) -- (2) (3) -- (4);
  \foreach \nn in {0,...,5}{\node[node] () at (\nn) {};}#4\end{scope}}
\newcommand{\RRrule}[3]{\begin{scope}[shift={(#1)},scale=#2,rotate=#3]
    \Rrule{0,0}{1}{  0}{}
    \Rrule{6,0}{1}{180}{}
    \Rrule{6,0}{1}{  0}{}
    \Rrule{3,3}{1}{270}{}
    \Rrule{6,0}{1}{ 90}{}\end{scope}}
    
\newcommand{\Rrulea}[4]{\begin{scope}[shift={(#1)},scale=#2,rotate=#3]
  \foreach \na in {0,...,4}{\coordinate (\na) at (72*\na-162:1.577) {};}
  \draw[draw=red,>=stealth,->] (0) -- (1);
  \draw[draw=red,>=stealth,->] (1) -- (2);
  \draw[draw=red,>=stealth,->] (4) -- (3);
  \draw[draw=myaqua!90!black,>=stealth,->] (3) -- (2);
  \draw[draw=myaqua!90!black,>=stealth,->] (0) -- (4);
  \foreach \na in {1,3,4}{\fill[black] (\na) circle (1.5pt);}
  \foreach \na in {0,2}{\fill[black] (\na) circle (2.25pt);}#4\end{scope}}

\newcommand{\Rruleb}[4]{\begin{scope}[shift={(#1)},scale=#2,rotate=#3]
\foreach \na in {0,...,5}{\pgfmathsetmacro\x{2.12*cos(18*\na+45)}
\pgfmathsetmacro\y{2.12*sin(18*\na+45)-1.5}
\coordinate (\na) at (\x,\y) {};}
\foreach \na in {6,...,9}{\pgfmathsetmacro\x{2.12*cos(18*\na+135)}
\pgfmathsetmacro\y{2.12*sin(18*\na+135)+1.5}
\coordinate (\na) at (\x,\y) {};}
  \draw[draw=red,>=stealth,->] (9) -- (0);
  \draw[draw=red,>=stealth,->] (1) -- (2);
  \draw[draw=red,>=stealth,->] (3) -- (2);  
  \draw[draw=red,>=stealth,->] (3) -- (4);
  \draw[draw=red,>=stealth,->] (4) -- (5);
  \draw[draw=myaqua!90!black,>=stealth,->] (0) -- (1); \draw[draw=myaqua!90!black,>=stealth,->] (5) -- (6);
  \draw[draw=myaqua!90!black,>=stealth,->] (7) -- (6);
  \draw[draw=myaqua!90!black,>=stealth,->] (8) -- (7); \draw[draw=myaqua!90!black,>=stealth,->] (8) -- (9);
  \foreach \na in {0,...,9}{\fill[black] (\na) circle (1.5pt);}
  \foreach \na in {0,5}{\fill[black] (\na) circle (2.25pt);}#4\end{scope}}
  
\newcommand{\RanRrule}[5]{\begin{scope}[shift={(#1)},scale=#2,rotate=#3]#4
\fill[black] (0,0) circle (2.25pt);\draw (0,-1) node(){};
\fill[black] (3,0) circle (2.25pt);#5\end{scope}}

\newcommand{\Raa}[4]{\RanRrule{#1}{#2}{#3}{\redarc{1.5,-0.5}{0,0}\redarc{3,0}{1.5,-0.5}
\bluearc{0,0}{1.5,0.5}\bluearc{1.5,0.5}{3,0}}{#4}}
\newcommand{\Rab}[4]{\RanRrule{#1}{#2}{#3}{\redarc{0,0}{1.5,0}\redarc{1.5,0}{3,0}\bluearc{1.5,-0.5}{0,0}\bluearc{3,0}{1.5,0.5} }{#4}}
\newcommand{\Rba}[4]{\RanRrule{#1}{#2}{#3}{\redarc{0,0}{1,0}\redarc{3,0}{2,0}\redarc{2,0.5}{3,0}\redarc{2,-0.5}{3,0}\redarc{1,-0.5}{0,0}
\bluearc{0,0}{1,0.5}\bluearc{1,0}{2,0}}{#4}}
\newcommand{\Rbb}[4]{\RanRrule{#1}{#2}{#3}{\redarc{1,0.5}{0,0}\redarc{2,0}{1,0}
\bluearc{0,0}{1,-0.5}\bluearc{1,0}{0,0}\bluearc{2,0}{3,0}\bluearc{3,0}{2,0.5}\bluearc{3,0}{2,-0.5}}{#4}}
\newcommand{\ABnodes}{\draw (0,-.25) node[below] {\makebox(0,0){\footnotesize $A$}};\draw (3,-.25) node[below] {\makebox(0,0){\footnotesize $B$}};}

\begin{figure}[ht]
\centering
\begin{tikzpicture}[shift={(0,0)},scale=.5,black,thick]
  \tikzstyle{node}=[circle,fill=white!75!black,draw=black ,inner sep = 0.25mm, outer sep = 0mm]
  \begin{scope}[shift={(-12.125,0)},scale=1]
    \draw[very thick] (0,0) -- (1,0);
    \mynode{0,0}{5}
    \mynode{1,0}{5}
  \end{scope}
  \Rrule{-10,0}{.8}{0}{}
  \RRrule{-6.5,0}{.6}{0}
  \begin{scope}[shift={(0,0)},scale=.4]
    \RRrule{0,0}{1}{  0}{}
    \draw (-28.75,-5) node {$G^0$};
    \RRrule{18,0}{1}{180}{}
    \draw (-21.75,-5) node {$G^1$};
    \RRrule{18,0}{1}{  0}{}
    \draw (-9.5,-5) node {$G^2$};
    \RRrule{9,9}{1}{270}{}
    \draw (13.75,-5) node {$G^3$};
    \RRrule{18,0}{1}{ 90}{}\end{scope}
\end{tikzpicture}
\caption{A substitution network}
\label{fig:substitution}
\end{figure}
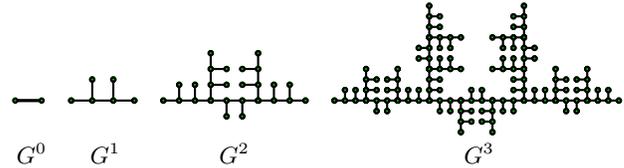

Xi et al.~\cite{XiWaWaYuWa17} proved that substitution networks have the scale-free property, as well as the related {\em fractality property}.
Li et al.~\cites{LiYaWa19,LiYuXi18} proved that substitution networks retain the scale-free property in the more general setting of colored arcs and fixed networks to replace each arc of a given color.
They also showed that certain of these colored substitution networks have the fractality property.
Similarly, substitutions of nodes, rather then arcs, were considered in~\cites{XiSuYa19,YaSuXi19}. 
By constructing self-similar networks, Yao et al.~\cite{YaSuXi19} proved that node substitution networks have the fractality property; see also~\cites{XiYe19b,YeGuXi19,YeXi19a}.

\bigskip

Although these substitution networks are suitable for modelling many real-world fractal-like phenomena, 
they are deterministic and do not reflect real-life randomness.

The purpose of this present paper is to address this limitation.
In particular, random colored substitution networks and their degree dimension are defined, 
and the main result of the paper, Theorem~\ref{thm:random scale-free}, 
extends the scale-freeness of substitution networks to random colored substitution networks.

\section{Random colored substitution networks}
\label{sec:random colored substitution networks}

Consider a directed network $G^0$ whose arcs are each colored by one of $\lambda$ colors.
Replace each arc of~$G^0$ according to its color as follows:
if arc $e$ has color~$i$,
then $e$ will be replaced randomly by a directed network $R_{ik}$ with probability $p_{ik}$,
among $q_i$ such directed networks (so $\sum_{k=1}^{q_i} p_{ik} = 1$ for each~$i$).
Each directed network $R_{ik}$ has a node $A$ and a node $B$
that respectively replace the beginning node $A$ and ending node $B$ of $e$;
this determines exactly how $R_{ik}$ replaces~$e$.
By replacing all arcs in $G^0$ randomly by the directed networks $R_{ik}$,
a directed network $G^1$ is obtained.
This replacement process iteratively defines a directed network $G^2$ from $G^1$,
a directed network $G^3$ from $G^2$, and so on.
After $t$ such iterations, 
a directed network $G^t = \big(V(G^t),E(G^t)\big)$ is obtained.
These graphs $R_{ij}$ are called {\em rule graphs}.

This iterative process and the resulting directed networks together form a {\em random colored substitution network}.
Throughout this paper, 
let $\mathcal{G}$ be the family of all possible sequences $\Gamma = (G^0, G^1, G^2, \ldots)$.
When such a sequence $\Gamma$ converges to a network, then we can identify $\Gamma$ as that network, 
together with the information on how it was generated.
That is, $\lim_{t\to\infty} G^t=\Gamma$.

An example of random colored substitution networks with arcs of $\lambda = 2$ colors arcs is given in Fig.~\ref{fig:rule-example}.
The red ($i=1$) arcs are each replaced at each step by 
the directed network $R_{11}$ with probability $\frac{1}{3}$ 
and the directed network $R_{12}$ with probability $\frac{2}{3}$. 
Each blue ($i=2$) arc is replaced either by $R_{21}$ or by $R_{22}$, 
with probabilities $\frac{1}{4}$ and $\frac{3}{4}$, respectively. 
One possible sequence $G^0, G^1, G^2, \ldots$ of the random colored substitution network is shown.

Note that (non-random) colored substitution networks form
the particular subclass of random colored substitution networks for which  $q_i = 1$ for all $i$.
Note also that it will be assumed in this paper that, for each color $i$, 
there is at least one integer $k$ and one integer $k'$ such that 
the distance between nodes $A$ and $B$ in the network $R_{ik}$ is greater than 1 and that, 
in $R_{ik'}$, 
the sum of the in-degree and the out-degree of at least one of the nodes $A$ and $B$ is greater than~1. 

\begin{remark}\label{rem:infinite}{\rm
These conditions ensure that the number of nodes in the substitution network and their degrees grow to infinity; 
this is proved in Section~\ref{sec:proof of random}.}
\end{remark}

\begin{figure}[ht]
\centering
\begin{tikzpicture}[scale=.75,shorten >= 1pt]
  \begin{scope}[scale=1,shift={(0,2)}]
    \bluearc{-1.25,0}{.75,0}
    \draw  (-1.25,0) node[below] {\footnotesize $A$};
    \draw  (  .75,0) node[below] {\footnotesize $B$};
    \draw[thick,gray,>=stealth,->] (1.5,0) .. controls (2.5,0) and (2.5, .6) .. (3.5, .6) node[black,sloped,midway,above]{\footnotesize$p_{11} = \frac{1}{3}$};
    \draw[thick,gray,>=stealth,->] (1.5,0) .. controls (2.5,0) and (2.5,-.6) .. (3.5,-.6) node[black,sloped,midway,below]{\footnotesize$p_{12} = \frac{2}{3}$};
    \Raa{4, .57}{.67}{0}{\ABnodes\draw (4,0) node {\makebox(0,0){$R_{11}$}};}
    \Rab{4,-.6 }{.67}{0}{\ABnodes\draw (4,0) node {\makebox(0,0){$R_{12}$}};}
  \end{scope}
  \begin{scope}[scale=1,shift={(0,-1)}]
    \redarc{-1.25,0}{.75,0}
    \draw   (-1.25,0) node[below] {\footnotesize $A$};
    \draw   (  .75,0) node[below] {\footnotesize $B$};
    \draw[gray,thick,>=stealth,->] (1.5,0) .. controls (2.5,0) and (2.5, .6) .. (3.5, .6) node[black,sloped,midway,above]{\footnotesize$p_{21} = \frac{1}{4}$};
    \draw[gray,thick,>=stealth,->] (1.5,0) .. controls (2.5,0) and (2.5,-.6) .. (3.5,-.6) node[black,sloped,midway,below]{\footnotesize$p_{22} = \frac{3}{4}$};
    \Rba{4,  .6}{.67}{0}{\ABnodes\draw (4,0) node {\makebox(0,0){$R_{21}$}};}
    \Rbb{4,-.6 }{.67}{0}{\ABnodes\draw (4,0) node {\makebox(0,0){$R_{22}$}};}
  \end{scope}
\end{tikzpicture}\\[5mm]
\begin{tikzpicture}[scale=.55]
  \bluearc{0,0}{2,0}
  \draw (1,-.5) node[below] {$G^0$};
  \draw[gray,thick,>=stealth,->] (2.5,0) --++ (1,0);
  \Raa{4,0}{.7}{0}{\draw (1.5,-.75) node[below] {$G^1$};}
  \draw[gray,thick,>=stealth,->] (6.5,0) --++ (1,0);
  \begin{scope}[scale=1,shift={(6.5,0)}]
    \Raa{1.5,0}{1}{30}{}
    \Rab{4.1,1.5}{1}{-30}{}
    \Rba{6.7,0}{1}{-150}{}
    \Rbb{4.1,-1.5}{1}{150}{}
    \fill[black] (1.5,0) circle (2.25pt);
    \fill[black] (6.7,0) circle (2.25pt);
    \draw (3.75,-1.5) node[below] {$G^2$};
  \end{scope}
  \draw[gray,thick,>=stealth,->] (14,0) --++ (1,0);
  \draw (15.75,0) node(){\large$\cdots$\!\!};
\end{tikzpicture}
\caption{A random colored substitution network}
\label{fig:rule-example}
\end{figure}
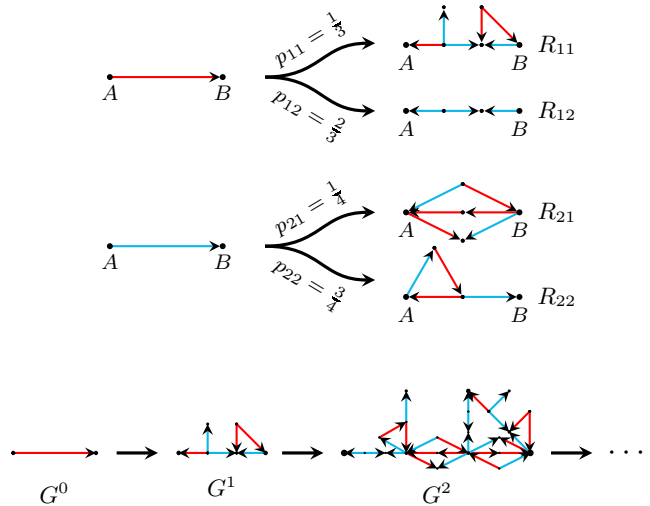

\section{Defining the scale-free property and the degree dimension}
\label{sec:defining-scale-freeness-and-degree-dimension}
\bigskip

For any undirected graph $G = \big(E(G),V(G)\big)$, 
let $\Delta(G)$ be the maximal degree of any node $v$ of $G$
and define the {\em normalised degree} of any node $v\in V$ to be 
$\hat{\deg}_G(v) = \frac{\deg(v)}{\Delta(G)}$.
This definition is extended to directed networks $G$ by letting $\hat{\deg}_G(v)$ be given by the underlying undirected graph of~$G$; 
that is, the graph obtained from~$G$ by ignoring arc directions.

\begin{definition}
Consider a network sequence $(G^0,G^1,G^2,\ldots)\in\mathcal{G}$ that converges to a network limit $\Gamma$. 
A node $v\in V(\Gamma)$ is {\em stationary} if the limit 
$\displaystyle\lim_{t\to\infty}\hat{\deg}_{G^t}(v)$ exists,
in which case, denote it by $\hat{\deg}_\Gamma(v)$.
The network sequence $\Gamma$ is (almost) {\em stationary} 
if (almost every) node of $\Gamma$ are stationary.
We will only consider almost stationary network sequences $\Gamma\in\mathcal{G}$ in this paper.

Now for each positive real number $\ell$, define 
\begin{equation}\label{equ:Pldef}
  P_\ell (\Gamma)  = \bigl| \{v\in V(\Gamma) \::\: \hat{\deg}_\Gamma(v) = \ell\}\bigr|\,.
\end{equation}
\end{definition}

\begin{definition}{\label{def:Scale-free Property}
Define
\[
\overline{\dim}_D (\Gamma)=\underset{\ell \to 0}{\lim\sup}\, \frac{\log P_\ell(\Gamma)}{-\log \ell}\,
\]
where the limit is taken over all values of $\ell$ such that $P_\ell(\Gamma)>0$.
The graph $\Gamma$ is {\em scale-free} if and only if, 
taking the limit $\ell\to\,0$ for all $\ell$ such that $P_\ell(\Gamma)>0$, 
\begin{equation}\label{eq:scale-free 1}
\dim_D(\Gamma) = \lim_{\ell\to\,0} \frac{\log P_\ell(\Gamma)}{-\log \ell}
\end{equation}
exists and is positive, in which case the limit is called the {\em degree dimension} of~$\Gamma$.}
\end{definition}

\begin{remark}
\label{rem:definition-more-general}{\rm
It is interesting, and potentially useful, to note that this definition can be used to define the scale-free property of many other infinite networks besides those arising from random colored substitution networks.}
\end{remark}

\begin{lemma}
If $\Delta (\Gamma) < \infty$, then $\dim_D (\Gamma)$ does not exist.
If $|V(\Gamma)|<\infty$, then $\dim_D (\Gamma)=0$.
\end{lemma}

\begin{proof}
If $\Delta (\Gamma) < \infty$, 
then no node has normalised degree $\ell$ for any $\ell<\frac{1}{\Delta(\Gamma)}$.
Given that $P_\ell(\Gamma)$ has to be positive, such limit does not exist accordingly.
If $|V(\Gamma)|<\infty$, 
then $P_\ell(\Gamma)\leq |V(\Gamma)|<\infty$. 
It follows that $\displaystyle\lim_{\ell\to\,0} (\log P_\ell(\Gamma))/(-\log \ell) = 0$.
\end{proof}

Therefore, $\Gamma$ is scale-free only when $\Delta (\Gamma) = \infty$ and $|V(\Gamma)| = \infty$. 
By Remark~\ref{rem:infinite}, the substitution networks in this paper will all feature networks $\Gamma$ with infinitely many arcs and $\Delta(\Gamma)=\infty$.

\begin{notation}
Let $\sim$ be defined as asymptotic equivalence.
Write $f(x) \overset{x \to x_0}{\asymp} g(x)$ whenever
$\displaystyle\lim_{x \to x_0} f(x)/g(x) = c$ for some constant $c>0$.
\end{notation}

\begin{lemma}
\label{lem:Sufficient Condition for Scale-free Property}
$\Gamma$ is scale-free if, taking the limit $\ell\to 0$ for all $\ell>0$ such that $P_\ell(\Gamma)>0$,
\begin{equation}\label{eq:scale-free 2}
  P_\ell(\Gamma)
    \overset{\ell\to\,0}{\asymp}\ell^{-\dim_D (\Gamma)}\,.
\end{equation}
\end{lemma}

\begin{proof}
(\ref{eq:scale-free 2}) implies $\displaystyle\dim_D(\Gamma)
=\lim_{\ell \to 0}\frac{\log c+ \log P_{\ell}(\Gamma)}{\log \ell}
=\lim_{\ell \to 0}\frac{\log P_{\ell}(\Gamma)}{\log \ell}$.
\end{proof}

\smallskip

\begin{notation}
For all $t$, define 
\[
  P_L(G^t) = \bigl| \{v\in V(G^t) \::\: \deg_{G^t}(v) = L\}\bigr|\,.
\]
\end{notation}

\smallskip

\begin{theorem}\label{thm:Sufficient Condition for Scale-free Property 2}
If $|V(G^t)|\overset{t \to \infty}{\asymp} \Delta(G^t)^{\delta}$,
then $\Gamma$ is scale-free with degree dimension $\delta=\dim_D(\Gamma)$ if, 
for all functions $L:\mathbb{N}\to\mathbb{N}$ satisfying 
$L(t)/\Delta (G^t) \overset{t \to \infty}{\sim}0$,
\begin{equation}\label{eq:scale-free 3}
  \frac{P_{L(t)}(G^t)}{\bigl|V(G^t)\bigr|} \overset{t \to \infty}{\asymp} L(t)^{\delta}\,.
\end{equation}
\end{theorem}

\begin{proof}
Suppose that Condition~(\ref{eq:scale-free 3}) holds.
Let $L(t)$ be any function on $\mathbb{N}$ 
satisfying $L(t)/\Delta (G^t) \overset{t \to \infty}{\sim}0$
and define the function $\ell(t)$ by $\ell(t)=L(t)/\Delta(G^t)$.
Then 
\[
  \frac{P_{\ell(t)}(G^t)}{\bigl|V(G^t)\bigr|} \overset{t \to \infty}{\asymp} \ell(t)^{-\delta} \Delta(G^t)^{-\delta}\,.
\]
Now $\Delta(G^t)^{\delta} \overset{t \to \infty}{\asymp} |V(G^t)|$, 
so  $P_{\ell(t)}(G^t) \overset{t \to \infty}{\asymp} \ell(t)^{-\delta}$.
Note that a sequence converges if and only if every subsequence of it converges. 
Hence, if all possible sequences $L(t)$ satisfy $
L(t)/\Delta(G)\overset{t \to \infty}{\to} 0$ and Condition~(\ref{eq:scale-free 3}), 
then Condition~(\ref{eq:scale-free 2}) will hold.
Therefore, 
\[
  P_{\ell}(\Gamma) \overset{\ell \to 0}{\asymp} \ell^{-\delta}\,,
\]
and so $\log P_{\ell}(\Gamma) \overset{\ell \to 0}{\sim} (-\log \ell) \delta$\,.
Therefore, $\Gamma$ is scale-free with $\dim_D (\Gamma)=\delta$.
\end{proof}

\begin{proposition}
Let $G_1$ and $G_2$ be two subnetworks of $G$.
Then
\[
    \overline{\dim}_D(G_1\cup G_2)
  = \max\bigl\{\overline{\dim}_D(G_1),\overline{\dim}_D(G_2)\bigr\} \,.
\]
\end{proposition}

\begin{proof}
For any $\ell$, 
\[
  2\max\{ P_\ell(G_1),P_\ell(G_2) \} \geq P_\ell(G_1 \cup G_2) \geq \max\{ P_\ell(G_1),P_\ell(G_2) \}\,.
\]
As a result, 
\begin{align*}
       \underset{\ell\to\,0}{\overline\lim}\,\frac{\log  \max\{P_\ell(G_1),P_\ell(G_2)\}}{-\log\ell} 
 &\leq \overline{\dim}_D(G_1\cup G_2)\\
 &\leq \underset{\ell\to\,0}{\overline\lim}\,\frac{\log 2\max\{P_\ell(G_1),P_\ell(G_2)\}}{-\log\ell}\,,
\end{align*}
while both sides converge to $\max\bigl\{\overline{\dim}_D(G_1),\overline{\dim}_D(G_2)\bigr\}$.

By induction, this property holds for finite unions as well.
\end{proof}

\section{Stochastic substitution processes}
\label{sec:stochastic substitution process}

\noindent
This section introduces a mathematical framework, called a {\em stochastic substitution process}. 
This framework provides the results that are applied in 
Sections~\ref{sec:Cardinality of degree} and~\ref{sec:proof of random}
to study the asymptotic properties of random colored substitution networks regarding Lyapunov exponents.

\medskip

\begin{notation}
For any vector $\mathbf{x}$, let $[\mathbf{x}]_i$ be the $i$-th entry of $\mathbf{x}$. 
For each network $G$ with arcs colored in colors $1,\ldots,\lambda$, 
define $\mathbf{\boldsymbol{\chi}}(G)$ to be the vector whose $j$-th entry is 
the number of $j$-colored arcs in~$G$.
Let $\Vert\mathbf{x}\Vert_1$ denote the sum of entries in~$\mathbf{x}$. 
For any vectors $\mathbf{x}$ and $\mathbf{y}$ of equal dimension,
write $\mathbf{x}\geq\mathbf{y}$ if $[\mathbf{x}]_i\geq[\mathbf{y}]_i$ for all~$i$. 
For any real square matrix $\mathbf{X}$, 
let $\rho(\mathbf{X})$ denote the {\em spectral radius} of~$\mathbf{X}$.

\medskip

Let $\mathcal{X} = \mathbf{X}_1,\dots,\mathbf{X}_N$ be a set of finitely many non-negative square matrices.
The set $\mathcal{X}$,
together with a probability vector $(p_1,\dots,p_N)$ where 
$\Pr(\mathbf{X}_i)=p_i$ for $i=1,\ldots,N$,
is called a {\em random matrices set}.
The notation $\Pr_\mathcal{X}(\mathbf{X}_i)$ 
will be used to denote $\Pr(\mathbf{X}_i)$, 
to highlight that these probabilities are associated with $\mathcal{X}$.
\end{notation}

Write $\mathcal{L}(\mathcal{X})$ as the maximal Lyapunov exponent of $\mathcal{X}$ defined by
\[
     \mathcal{L}(\mathcal{X})
  := \lim_{n \to \infty} \frac{1}{n} \mathbb{E} 
      \bigl(\log \Vert \mathbf{Y}_{i_1} \cdots \mathbf{Y}_{i_n} \Vert\bigr) \,,
\]
where $\mathbf{Y}_{i_k} \in \mathcal{X}$ is chosen with probability $\Pr_\mathcal{X}(\mathbf{Y}_{i_k})$, 
and where $\mathbb{E}(*)$ is the expectation value.
The study of asymptotic behaviours of random matrices product dates back to 
Bellman~\cite{Bellman54}, Furstenberg and Kesten~\cites{Furstenberg60,Furstenberg63}, 
Guivarc~\cite{Guivarc85} and Le Page~\cite{Page82}.
A famous theorem by Furstenberg and Kesten~\cite{Furstenberg60} 
asserts that $\mathcal{L}$ exists and that
\[
  \mathcal{L}(\mathcal{X})
  \overset{a.s.}{=\joinrel=}\lim_{n \to \infty} 
  \frac{1}{n} \log\Vert\mathbf{Y}_{i_1} \cdots \mathbf{Y}_{i_n} \Vert 
  \overset{a.s.}{=\joinrel=}\lim_{n \to \infty} 
  \frac{1}{n} \log[\mathbf{Y}_{i_1} \cdots \mathbf{Y}_{i_n} ]_{jk}
\]
if all $\mathbf{X}\in\mathcal{X}$ are primitive.

\begin{definition}
Let $\mathcal{X}=\{\mathbf{X}_1,\dots,\mathbf{X}_N\}$ be a random matrices set.
For each $i=1,\ldots,m$, 
let $\mathbf{e}_{i}\in\mathbb{R}^m$ be the $i$-th standard basis unit vector of~$\mathbb{R}^m$.
Define a random function 
$T_{\mathcal{X}}^{'} \::\: \{ \mathbf{e}_{1}, \dots, \mathbf{e}_{m}\} \to (\mathbb{Z}^+)^m$
by setting $m$ identical and independent random vectors $T_{\mathcal{X}}^{'}(\mathbf{e}_i)$,
each with probability
\[
  \Pr(T^{'}_{\mathcal{X}}(\mathbf{e}_{i}) = \mathbf{e}_{i} \mathbf{X}_j) = p_j \,.
\]

We decompose all 
$\mathbf{x} = x_1\mathbf{e}_1+\cdots+x_m\mathbf{e}_m \in(\mathbb{Z}^+)^m$ 
($x_i \in \mathbb{Z}$) 
through the random function $T_{\mathcal{X}} \::\: (\mathbb{Z}^+)^m  \to (\mathbb{Z}^+)^m$ by
\[
    T_{\mathcal{X}}(\mathbf{x}) 
  = \sum_{j=1}^\lambda \sum_{i=1}^{x_j} T^{'}_{\mathcal{X}}(\mathbf{e}_j)\,.
\]
For simplicity, write 
$T_{\mathcal{X}}^n = \overbrace{T_{\mathcal{X}} \circ \cdots \circ T_{\mathcal{X}}}^{n}$.
We call such $T_{\mathcal{X}}^n(\mathbf{x})$ a {\em stochastic substitution process}. 
That is because the decomposition of $\mathbf{x}$ represents the independent substitution of each arc, 
and $T^{'}_{\mathcal{X}}$ indicates the result of substitution.
\end{definition}

\begin{theorem}\label{thm:stochastic substitution process}
\[
\lim_{n\to \infty} \frac{1}{n} \log \Vert T_{\mathcal{X}}^n(\mathbf{x}_0) \Vert =\mathcal{L}(\mathcal{X}) \quad a.s.
\]
\end{theorem}
\begin{proof}
Note that $\mathbf{Y}_{i_1},\ldots,\mathbf{Y}_{i_n}$ forms a stationary stochastic process.
By Furstenberg and Kesten's theorem~\cite{Furstenberg60},
$\displaystyle\lim_{n\to \infty} \frac{1}{n} \log \mathbb{E}\bigl(\Vert T_{\mathcal{X}}^n(\mathbf{x}_0)\bigr) \Vert \bigr)$
exists.
Hence, for any~$\mathbf{x}_0$,
\begin{align*}
\lim_{n\to \infty} \frac{1}{n} \log \Vert T_{\mathcal{X}}^n(\mathbf{x}_0) \Vert 
&=\lim_{n\to \infty} \frac{1}{n} \log \mathbb{E}\bigl(\Vert T_{\mathcal{X}}^n(\mathbf{x}_0)\bigr) \Vert \bigr) \\
&=\lim_{n \to \infty} \frac{1}{n} \mathbb{E} \bigl(\log \Vert \mathbf{x}_0 \mathbf{Y}_{i_1} \cdots \mathbf{Y}_{i_n} \Vert\bigr) \\
&=\lim_{n \to \infty} \frac{1}{n} \log \Vert \mathbf{Y}_{i_1} \cdots \mathbf{Y}_{i_n} \Vert 
\overset{a.s.}{=\joinrel=}\mathcal{L}(\mathcal{X}) \qedhere
\end{align*}
\end{proof}
This theorem reveals that the growth rate of stochastic substitution process almost surely follows the Lyapunov exponent of the random matrices product.

\smallskip

\begin{lemma}\label{lem:Norm Limit}
Let $\mathbf{X}$ be a primitive non-negative $n \times n$ matrix with spectral radius $\rho(\mathbf{X})$.
Then, for any positive vector $\mathbf{u}\in(\mathbb{R}^+)^n$,
\[
  \Vert \mathbf{u}\mathbf{X}^t\Vert_1\overset{t\to \infty}{\asymp} \rho(\mathbf{X})^t\,.
\]
\end{lemma}
\begin{proof}
As $\mathbf{X}$ is primitive, 
the Perron-Frobenius Theorem implies that the following limit matrix exists and is positive:
$\displaystyle\lim_{t \to \infty} \Bigl(\frac{\mathbf{X}}{\rho(\mathbf{X})}\Bigr)^t = \mathbf{G}$.
Hence,
\[  
  \bigg\Vert  \mathbf{u} \lim_{t \to \infty} \biggl(\frac{\mathbf{X}}{\rho(\mathbf{X})}\biggr)^t \bigg\Vert_1 
= \Vert  \mathbf{u} \mathbf{G} \Vert_1
= c
\]
where $c>0$ is a constant depending on $\mathbf{u}$ and $\mathbf{X}$.
Thus,
$\Vert\mathbf{u}\mathbf{X}^t\Vert_1\overset{t\to \infty}{\asymp} \rho(\mathbf{X})^t$.
\end{proof}

\begin{lemma}\label{lem:random to deterministic}
Let $\mathcal{X}$ be a random matrices set.
If $\mathcal{X}=\{\mathbf{X}\}$, 
then $\mathcal{L}(\mathcal{X})=\log\rho(\mathbf{X})$.
\end{lemma}
\begin{proof}
By definition and Lemma~\ref{lem:Norm Limit}
\[
\mathcal{L}(\mathcal{X})
=\lim_{n \to \infty} \frac{1}{n} \log \Vert \mathbf{X}^n\Vert 
=\lim_{n \to \infty} \frac{1}{n} \log c\rho(\mathbf{X})^n
=\log \rho(\mathbf{X})\,. \qedhere
\]
\end{proof}

\section{Main results}
\label{sec:main}

The main result of this paper is Theorem~\ref{thm:random scale-free} which states that 
random colored substitution networks are scale-free under certain natural conditions.

\begin{definition}\label{def:arc matrix}
For each $i_j \in \{1,\ldots,q_j\}$, define 
\[
\mathbf{M}=
\begin{pmatrix}
\boldsymbol{\chi}(R_{1i_1}) \\
\vdots \\
\boldsymbol{\chi}(R_{\lambda i_\lambda})
\end{pmatrix}
\,.
\]
Then collect all $\mathbf{M}$ to obtain
\[
\mathcal{M}=\bigl\{\mathbf{M}\::\: i_j \in \{1,\ldots,q_j\},j\in\{ 1,\ldots,\lambda\} \bigr\} \,.
\]
Note that $\mathcal{M}$ has $\prod_{j=1}^\lambda q_j$ elements 
and  that $\Pr_\mathcal{M}(\mathbf{M})=\prod_{j=1}^\lambda p_{ji_j}$.
Note also that $\mathcal{M}$ is a random matrices set
since $\sum_{\mathbf{M} \in \mathcal{M}} \Pr_\mathcal{M}(\mathbf{M})=1$.
In this paper, we assume that all matrices in $\mathcal{M}$ are primitive.
\end{definition}
\begin{notation}
For a node $v$ of an arc-colored directed network $G$,
let $\deg_j^+(G:v)$ and $\deg_j^-(G:v)$ denote
the number of $j$-colored out-going arcs $(v,w)$
and the number of $j$-colored in-going arcs $(u,v)$ in~$G$, respectively.
Let 
\[
\boldsymbol{\delta}(G:v):=\bigl(\deg_1^+ (G:v), \deg_1^- (G:v),\ldots,\deg_\lambda^+ (G:v), \deg_\lambda^- (G:v)\bigr) 
\]
be a $2\lambda$-dimensional non-negative vector.
\end{notation}

\begin{definition}\label{def:degree matrix}
Define the $2\lambda \times 2\lambda$ matrix
\[
  \mathbf{N} = 
\begin{pmatrix}
\boldsymbol{\delta}(R_{1i_1}:A) \\
\boldsymbol{\delta}(R_{1i_1}:B) \\
\vdots \\
\boldsymbol{\delta}(R_{\lambda i_\lambda}:A) \\
\boldsymbol{\delta}(R_{\lambda i_\lambda}:B)
\end{pmatrix} \,,
\]
where $i_j\in \{1,\ldots,q_j\}$.

Let $\mathcal{N}$ be the set of these random matrices
\[
\mathcal{N}=\bigl\{\mathbf{N} \::\: i_j \in \{1,\ldots,q_j\},j\in\{1,\ldots,\lambda\} \bigr\}
\]
and note that $|\mathcal{N}|=\prod_{j=1}^\lambda q_j$. 
To assign probability for each matrix, let $\Pr_\mathcal{N}(\mathbf{N})=\prod_{j=1}^\lambda p_{j i_j}$.
In this way, $\mathcal{N}$ is a random matrices set.
Moreover, in this paper, all matrices in $\mathcal{N}$ are assumed to be primitive.
\end{definition}

\begin{example}\label{exa:M-and-N}
For the random colored substitution network of Fig.~\ref{fig:rule-example}, 
\[
\mathcal{M}=\biggl\{
\begin{pmatrix}
    2 & 2 \\
    2 & 5 
\end{pmatrix}
,
\begin{pmatrix}
    2 & 2 \\
    5 & 2
\end{pmatrix}
,
\begin{pmatrix}
    2 & 2 \\
    2 & 5
\end{pmatrix}
,
\begin{pmatrix}
    2 & 2 \\
    5 & 2
\end{pmatrix}
\biggr\}
\]
and 
\[
\mathcal{N}=\left\{
\begin{pmatrix}
    1 & 0 & 0 & 1 \\
    0 & 1 & 1 & 0 \\
    1 & 0 & 0 & 0 \\
    1 & 1 & 1 & 2
\end{pmatrix}
,
\begin{pmatrix}
    1 & 0 & 0 & 1 \\
    0 & 1 & 1 & 0 \\
    1 & 1 & 2 & 1 \\
    0 & 1 & 0 & 0
\end{pmatrix}
,
\begin{pmatrix}
    0 & 1 & 1 & 0 \\
    1 & 0 & 0 & 1 \\
    1 & 0 & 0 & 0 \\
    1 & 1 & 1 & 2
\end{pmatrix}
,
\begin{pmatrix}
    0 & 1 & 1 & 0 \\
    1 & 0 & 0 & 1 \\
    1 & 1 & 2 & 1 \\
    0 & 1 & 0 & 0
\end{pmatrix}
\right\}\,.
\]
Both $\mathcal{M}$ and $\mathcal{N}$ have associated probability vectors $(\frac{1}{12},\frac{1}{4},\frac{1}{6},\frac{1}{2})$.
\end{example}

The main result of the paper is as follows.
\begin{theorem}\label{thm:random scale-free}
For a random colored substitution network, 
almost every $\Gamma \in \mathcal{G}$ is stationary and scale-free with associated degree dimension
\[
\dim_D (\Gamma)
\overset{a.s}{=\joinrel=}\dfrac{\mathcal{L}(\mathcal{M})}{\mathcal{L}(\mathcal{N})} \,.
\]
\end{theorem}

Theorem~\ref{thm:random scale-free} will be proved in Section~\ref{sec:proof of random}.
\begin{corollary}\label{coro:deterministic scale-free}
For a {\em deterministic} colored substitution network,  
$\Gamma$ is stationary and scale-free with associated degree dimension
\[
\dim_D (\Gamma)
=\dfrac{\log\rho(\mathbf{M})}{\log\rho(\mathbf{N})}\,.
\]
\end{corollary}

\bigskip

Here we present a powerful application of random colored substitution networks to the analysis of the degree dimension.
\begin{proposition}
There is no fixed inequality between $\dim_D (G)$ and $\dim_D (H)$
that holds for $G \subset H$.
\end{proposition}

\begin{proof}

\begin{figure}
\centering
\begin{minipage}{0.3\textwidth}
\begin{tikzpicture}[scale=.55]
  \redarc{0,0}{2,0}
  \draw (0,0) node[below] {\footnotesize $A$};
  \draw (2,0) node[below] {\footnotesize $B$};
  \draw[gray,thick,>=stealth,->] (2.5,0) --++ (1,0);
  \draw (4,0) node[below] {\footnotesize $A$};
  \redarc{4,0}{5,0}
  \redarc{6,0}{5,0}
  \redarc{5,1}{4,0}
  \redarc{5,1}{5,0}
  \redarc{5,1}{6,0}
  \draw  (6,0) node[below] {\footnotesize $B$};
\end{tikzpicture}
Substitution network 1
\end{minipage}
\hfill
\begin{minipage}{0.3\textwidth}
\begin{tikzpicture}[scale=.55]
  \redarc{0,0}{2,0}
  \draw (0,0) node[below] {\footnotesize $A$};
  \draw (2,0) node[below] {\footnotesize $B$};
  \draw[gray,thick,>=stealth,->] (2.5,0) --++ (1,0);
  \draw (4,0) node[below] {\footnotesize $A$};
  \redarc{4,0}{5,0}
  \redarc{6,0}{5,0}
  \redarc{5,1}{4,0}
  \redarc{5,1}{6,0}
  \draw  (6,0) node[below] {\footnotesize $B$};
\end{tikzpicture}
Substitution network 2
\end{minipage}
\hfill
\begin{minipage}{0.3\textwidth}
\begin{tikzpicture}[scale=.55]
  \redarc{0,0}{2,0}
  \draw (0,0) node[below] {\footnotesize $A$};
  \draw (2,0) node[below] {\footnotesize $B$};
  \draw[gray,thick,>=stealth,->] (2.5,0) --++ (1,0);
  \draw (4,0) node[below] {\footnotesize $A$};
  \redarc{4,0}{5,0}
  \redarc{6,0}{5,0}
  \redarc{5,1}{4,0}
  \redarc{5,1}{5,0}
  \draw  (6,0) node[below] {\footnotesize $B$};
\end{tikzpicture}
Substitution network 3
\end{minipage}
\caption{Three examples of substitution networks}
\label{figure:three networks}
\end{figure}

Consider three substitution networks as shown in Fig.~\ref{figure:three networks}, 
where 
\[
\mathbf{M}_1=5\,,\;
\mathbf{M}_2=4\,,\;
\mathbf{M}_3=4\,;\;
\quad
\mathbf{N}_1=
\begin{pmatrix}
1 & 1 \\
1 & 1
\end{pmatrix}
,\;
\mathbf{N}_2=
\begin{pmatrix}
1 & 1 \\
1 & 1
\end{pmatrix}
,\;
\mathbf{N}_3=
\begin{pmatrix}
1 & 1 \\
1 & 0
\end{pmatrix}.
\]
Their degree dimensions are as follows, by Theorem~\ref{coro:deterministic scale-free}:
\begin{align*}
&\dim_D(\Gamma_1)
=\frac{\log \rho(\mathbf{M}_1)}{\log \rho(\mathbf{N}_1)}
=\frac{\log 5}{\log 2}
\approx 2.3219 \,,\\
&\dim_D(\Gamma_2)
=\frac{\log \rho(\mathbf{M}_2)}{\log \rho(\mathbf{N}_2)}
=\frac{\log 4}{\log 2}
=2 \,, \\
&\dim_D(\Gamma_3)
=\frac{\log \rho(\mathbf{M}_3)}{\log \rho(\mathbf{N}_3)}
=\frac{\log4}{\log (\frac{1}{2}\sqrt{5}+\frac{1}{2})}
\approx 2.8808 \,.
\end{align*}

Even though $\Gamma_2$ and $\Gamma_3$ are both subnetworks of $\Gamma_1$,
their degree dimensions are neither both less than, nor both greater than, that of $\Gamma_1$.
\end{proof}

\begin{example}\label{exa:fractality}
The random colored substitution network of Fig.~\ref{fig:rule-example} and Example~\ref{exa:M-and-N} has
the associated Lyapunov exponents 
$\mathcal{L}(\mathcal{M})\approx 1.6692$ and 
$\mathcal{L}(\mathcal{N})\approx 0.9349$.
Therefore, the associated degree dimension, 
as defined by Definition~\ref{def:Scale-free Property}, is
\[
  \dim_D(\Gamma) \overset{a.s.}{=\joinrel=} \frac{\mathcal{L}(\mathcal{M})}{\mathcal{L}(\mathcal{N})} \approx \frac{1.6692}{0.9349} \approx  1.7854 \,.
\]
We obtain ten sets of simulated values for when $t=5$; see Fig.~\ref{fig:Degreet}.
Note that the data is approximately linear on a log-log plot
and that the average degree dimension from these ten randomly simulated data sets is $1.7891$, 
which, despite the low value of $t$ and the small number of simulations, 
is close to the theoretical asymptotic value $\dim_D(\Gamma) \approx 1.7854$.
\end{example}

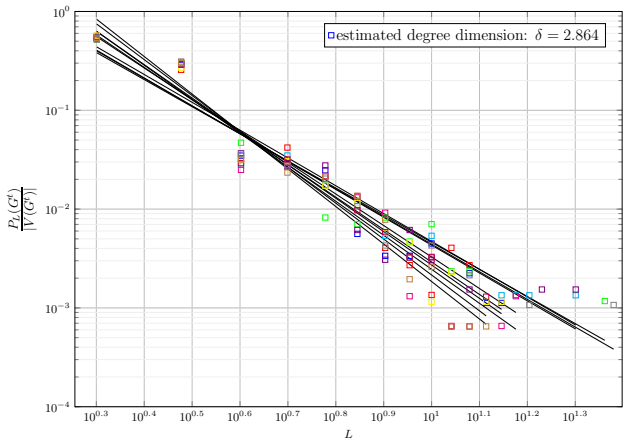
\begin{figure}[ht]
\centering
\hspace{-.5cm}
\begin{tikzpicture}[scale=.7]
\pgfplotstableread{
X1 Y1 X2 Y2 X3 Y3 X4 Y4 X5 Y5 X6 Y6 X7 Y7 X8 Y8 X9 Y9 X10 Y10
0.0116	0.495	0.0118	0.4905	0.013	0.5616	0.0125	0.5541	0.0119	0.5503	0.013	0.5203	0.0133	0.5426	0.0135	0.5604	0.012	0.552	0.0096	0.4899
0.0233	0.3844	0.0235	0.3926	0.026	0.4377	0.025	0.4365	0.0238	0.4212	0.026	0.4071	0.0267	0.4244	0.027	0.4256	0.0241	0.4224	0.0192	0.3891
0.0349	0.0268	0.0353	0.0254	0.039	0.0218	0.0375	0.022	0.0357	0.025	0.039	0.0281	0.04	0.0252	0.0405	0.0243	0.0361	0.0243	0.0288	0.0271
0.0465	0.0054	0.0471	0.006	0.0519	0.0057	0.05	0.007	0.0476	0.0072	0.0519	0.0047	0.0533	0.0079	0.0541	0.0077	0.0482	0.0063	0.0385	0.0056
0.0581	0.0454	0.0588	0.0468	0.0649	0.0464	0.0625	0.0474	0.0595	0.046	0.0649	0.0479	0.0667	0.0465	0.0676	0.0446	0.0602	0.0461	0.0481	0.0455
0.0698	0.0205	0.0706	0.0196	0.0779	0.0218	0.075	0.0207	0.0714	0.0204	0.0779	0.0167	0.08	0.0193	0.0811	0.0216	0.0723	0.0204	0.0577	0.0184
0.0814	0.003	0.0824	0.0048	0.0909	0.003	0.0875	0.0033	0.0833	0.0026	0.0909	0.0038	0.0933	0.0036	0.0946	0.002	0.0843	0.003	0.0673	0.0045
0.093	0.0021	0.0941	9.06E-04	0.1039	3.36E-04	0.1	6.68E-04	0.0952	0.0013	0.1039	0.0019	0.1067	9.82E-04	0.1081	0.0013	0.0964	9.87E-04	0.0769	0.0018
0.1163	6.02E-04	0.1176	6.04E-04	0.1299	0.001	0.125	0.001	0.119	0.0013	0.1299	0.0013	0.1333	0.0016	0.1351	0.0017	0.1205	0.0013	0.0962	0.0015
0.1279	0.0045	0.1294	0.0042	0.1429	0.0044	0.1375	0.005	0.131	0.0043	0.1429	0.005	0.1467	0.0043	0.1486	0.0037	0.1325	0.0036	0.1058	0.0048
0.1395	0.0024	0.1412	0.0027	0.1558	0.003	0.15	0.0027	0.1429	0.0026	0.1558	0.0022	0.16	0.0023	0.1622	0.002	0.1446	0.002	0.1154	0.0015
0.1512	0.0021	0.1529	0.0021	0.1688	0.003	0.1625	0.0037	0.1548	0.0023	0.1688	0.0022	0.1733	0.002	0.1757	0.0033	0.1566	0.003	0.125	0.0033
0.1628	0.0024	0.1647	0.0027	0.1818	0.0017	0.175	0.001	0.1667	0.0033	0.1818	0.0019	0.1867	0.0023	0.1892	0.0027	0.1687	0.0039	0.1346	0.0018
0.186	0.0015	0.1882	0.0015	0.2078	0.0017	0.2	0.0013	0.1905	3.29E-04	0.2078	0.0016	0.2133	0.0023	0.2162	0.0013	0.2169	3.29E-04	0.1538	0.0012
0.1977	3.01E-04	0.2118	3.02E-04	0.2208	3.36E-04	0.225	3.34E-04	0.2024	3.29E-04	0.2727	3.15E-04	0.32	3.27E-04	0.2297	3.33E-04	0.253	3.29E-04	0.1635	2.97E-04
0.2093	3.01E-04	0.3176	6.04E-04	0.3636	3.36E-04	0.325	6.68E-04	0.2143	3.29E-04	0.3896	9.46E-04	0.3467	3.27E-04	0.2973	3.33E-04	0.2892	3.29E-04	0.1731	2.97E-04
0.3023	3.01E-04	0.3647	9.06E-04	0.3766	6.72E-04	0.35	6.68E-04	0.2262	3.29E-04	0.4026	3.15E-04	0.36	3.27E-04	0.3514	3.33E-04	0.3253	0.0013	0.2115	2.97E-04
0.314	3.01E-04	0.3882	3.02E-04	0.4026	0.001	0.375	3.34E-04	0.3214	3.29E-04	0.4156	3.15E-04	0.44	3.27E-04	0.3649	3.33E-04	0.3494	3.29E-04	0.25	2.97E-04
0.3256	6.02E-04	0.4118	3.02E-04	0.7143	3.36E-04	0.3875	6.68E-04	0.3333	3.29E-04	0.4286	3.15E-04	0.4533	3.27E-04	0.3784	3.33E-04	0.3614	3.29E-04	0.2596	2.97E-04
0.3488	3.01E-04	0.4353	3.02E-04	0.8312	3.36E-04	0.475	3.34E-04	0.3571	3.29E-04	0.4416	3.15E-04	0.4667	3.27E-04	0.4054	3.33E-04	0.3735	6.58E-04	0.2692	2.97E-04
0.3605	6.02E-04	0.7882	3.02E-04	0.8571	3.36E-04	0.8125	6.68E-04	0.369	3.29E-04	0.5195	3.15E-04	0.5067	3.27E-04	0.4189	3.33E-04	0.7952	3.29E-04	0.2788	2.97E-04
0.4767	3.01E-04	0.9647	3.02E-04	1	3.36E-04	1	3.34E-04	0.4048	3.29E-04	0.8831	3.15E-04	0.7733	3.27E-04	0.4324	3.33E-04	0.8675	3.29E-04	0.2981	5.95E-04

}\mytable
    \begin{axis}[
        xmode=log,
        ymode=log,
        xmin = 0.01, xmax = 1,
        ymin = 0.0001, ymax = 1,
        width = 0.9\textwidth,
        height = 0.675\textwidth,
        grid = both,
        minor tick num = 1,
        major grid style = {lightgray},
        minor grid style = {lightgray!25},
        xlabel = {$\ell$},
        ylabel = {\Large$P_\ell(G^5)$},
        legend cell align = {left},
        legend pos = north east
        ]
        \addlegendentry{\large\;\,estimated degree dimension: $\dim_D (G^5) \approx 1.7891$}
        \addplot[color=blue,mark=square,only marks] table[x = X1, y = Y1] {\mytable};
        \addplot[thick, black] table[ x = X1, y = {create col/linear regression={y=Y1}}] {\mytable};
        \addplot[color=red,mark=square,only marks] table[x = X2, y = Y2] {\mytable};
        \addplot[thick, black] table[ x = X2, y = {create col/linear regression={y=Y2}} ] {\mytable};
        \addplot[color=green,mark=square,only marks] table[x = X3, y = Y3] {\mytable};
        \addplot[thick, black] table[ x = X3, y = {create col/linear regression={y=Y3}} ] {\mytable};
        \addplot[color=cyan,mark=square,only marks] table[x = X4, y = Y4] {\mytable};
        \addplot[thick, black] table[ x = X4, y = {create col/linear regression={y=Y4}} ] {\mytable};
        \addplot[color=magenta,mark=square,only marks] table[x = X5, y = Y5] {\mytable};
        \addplot[thick, black] table[ x = X5, y = {create col/linear regression={y=Y5}} ] {\mytable};
        \addplot[color=yellow,mark=square,only marks] table[x = X6, y = Y6] {\mytable};
        \addplot[thick, black] table[ x = X6, y = {create col/linear regression={y=Y6}} ] {\mytable};
        \addplot[color=violet,mark=square,only marks] table[x = X7, y = Y7] {\mytable};
        \addplot[thick, black] table[ x = X7, y = {create col/linear regression={y=Y7}} ] {\mytable};
        \addplot[color=gray,mark=square,only marks] table[x = X8, y = Y8] {\mytable};
        \addplot[thick, black] table[ x = X8, y = {create col/linear regression={y=Y8}} ] {\mytable};
        \addplot[color=purple,mark=square,only marks] table[x = X9, y = Y9] {\mytable};
        \addplot[thick, black] table[ x = X9, y = {create col/linear regression={y=Y9}} ] {\mytable};
        \addplot[color=brown,mark=square,only marks] table[x = X10, y = Y10] {\mytable};
        \addplot[thick, black] table[ x = X10, y = {create col/linear regression={y=Y10}} ] {\mytable};
    \end{axis}
\end{tikzpicture}
\caption{Simulations of scale-freeness for $t=5$}
\label{fig:Degreet}
\end{figure}

\section{Graph properties}
\label{sec:Cardinality of degree}

\begin{notation}
In this paper, set function $s(t):\mathbb{Z}^+ \to \mathbb{Z}^+$ such that $s(t) \leq t$ for all $t\in\mathbb{Z}^+$, and $s(t)\to \infty$.
Without ambiguity, we will simply denote $s(t)$ by $s$.
Let $V^*(G_s^t)$ be a subset of $V(G^t)$ so that $V^*(G_s^t):=V(G^s)\setminus V(G^{s-1})$.
Also, let $v^t_{s}$ denote a fixed node in $V^*(G_{s}^t)$. 

\end{notation}

\begin{theorem}\label{thm:Degree Growth}
Almost every $\Gamma \in \mathcal{G}$ satisfies 
\[
  \deg (v_{s}^t)\overset{t\to \infty}{\asymp} \exp(\mathcal{L}(\mathcal{N}))^{t-s}\,.
\]
\end{theorem}
\begin{proof}
When an $i$-colored arc connecting $v_{s}^{t}$ is substituted randomly according to the rule graphs, the result corresponds to $T_{\mathcal{N}}^{'}(\mathbf{e}_i)$.
By considering all arcs connecting $v_{s}^{t}$, we obtain 
\[
\deg(v_{s}^{t+1})=\Vert T_{\mathcal{N}}(\boldsymbol{\delta}(v_{s}^{t})) \Vert_1\,.
\]
Let $\mathbf{x}_0=\boldsymbol{\delta}(v_{s}^{s})$, and by induction
\[
\deg_{G^t}(v_{s}^t)=\Vert T_{\mathcal{N}}^{t-s}(\mathbf{x}_0) \Vert_1 \,.
\]
Finally by Theorem~\ref{thm:stochastic substitution process} we conclude that $
\deg (v_{s(t)}^t)\overset{t\to \infty}{\asymp} \exp(\mathcal{L}(\mathcal{N}))^{t-s(t)}$ almost surely.
\end{proof}

\begin{lemma}\label{lem:degree range}
    $c^{-1} \leq \frac{\deg(v_{s(t)}^t)}{\exp(\mathcal{L}(\mathcal{N}))^{t-s(t)}}\leq c$ almost surely as $t \to \infty$.
\end{lemma}
\begin{proof}
By Theorem~\ref{thm:Degree Growth}. 
\end{proof}

\begin{theorem}\label{thm:Arc Growth}
Almost every $\Gamma \in\mathcal{G}$
satisfies
\[
  |E(G^t)|\overset{t \to \infty}{\asymp} \exp(\mathcal{L(\mathcal{M})})^t\,.
\]
\end{theorem}
\begin{proof}
The essence of this proof is similar to that of Theorem~\ref{thm:Degree Growth}.
Whenever an $i$-colored arc is substituted by some rule graph, 
the result corresponds to $T_{\mathcal{M}}^{'}(\mathbf{e}_i)$.
Collecting all arcs in $G^t$, we have, for any fixed $G_0$,
\begin{align*}
|E(G^{t+1})|
=\Vert \boldsymbol{\chi}(G^{t+1})\Vert_1 
=\Vert T_{\mathcal{M}}(\boldsymbol{\chi}(G^t))\Vert_1 
=\Vert T^{t}_{\mathcal{M}}(\boldsymbol{\chi}(G^0))\Vert_1 \,.
\end{align*}

The proof follows by induction and Theorem~\ref{thm:stochastic substitution process}.
\end{proof}

\begin{theorem}\label{thm:Node Growth}
Almost every $\Gamma \in\mathcal{G}$ satisfies 
\[
  |V(G^t)|\overset{t \to \infty}{\asymp} \exp(\mathcal{L}(\mathcal{M}))^t\,.
\]
\end{theorem}
\begin{proof}
An important observation is
\[
  |V(G^t)|=\sum_{i=0}^{t} |V^*(G_i^t)|\,.
\]
Define random vectors set $\mathcal{V}$ by
\[
\mathcal{V}=\biggl\{ \mathbf{V}=
\begin{pmatrix}
  |V(R_{1 j_1})| -2 \\
  \vdots\\
  |V(R_{\lambda j_\lambda})| -2 
\end{pmatrix}
\::\: j_i \in 1,\ldots,q_i\,,\;i\in 1,\ldots,\lambda \biggr\} \,,
\]
with probability $P(\mathbf{V})=\prod_{i=1}^\lambda p_{ij_i}$. 
All new nodes $V^*(G_t^t)$ in $V(G^t)$ are generated by substituting arcs in $G^{t-1}$, 
so it follows that, for all $t\in \mathbb{N}$,
\begin{equation*} 
 |V^{*}(G_t^t)| = T_{\mathcal{V}}(\boldsymbol{\chi}(G^{t-1}))
\,.
\end{equation*}
Consequently, with $|V(G^0)|<\infty$ and by Furstenberg and Kesten's theorem~\cite{Furstenberg60}, 
$|V(G^t)|$ equals 
\begin{align*} 
\sum_{i=0}^{t} |V^*(G^{i})|
&=|V(G^0)|+\sum_{i=1}^{t} T_{\mathcal{V}}(\boldsymbol{\chi}(G^{i-1}))\\
&=|V(G^0)|+\sum_{i=1}^{t} T_{\mathcal{V}}(T^{i-1}_{\mathcal{M}}(\boldsymbol{\chi}(G^0))) \\
&=|V(G^0)|+ T_{\mathcal{V}}(\sum_{i=1}^{t} T^{i-1}_{\mathcal{M}}(\boldsymbol{\chi}(G^0))) \\
&\asymp |V(G^0)|+ \sum_{i=1}^{t} \exp(\mathcal{L}(\mathcal{M})) \\
&\asymp \exp(\mathcal{L}(\mathcal{M}))^t\,.
\end{align*} 
almost surely as $T_{\mathcal{V}}$ is bounded.
\end{proof}

\section{Proof of Theorem~\ref{thm:random scale-free}}
\label{sec:proof of random}

This section is devoted to proving the scale-freeness for random colored substitution networks.

\bigskip

\noindent{\sl Proof of Theorem~\ref{thm:random scale-free}}.
First, we prove that $\Delta(\Gamma)=\infty$ almost surely.
Recall that we assume all matrices in $\mathcal{M}$ and $\mathcal{N}$ to be primitive.
This implies $\min_{\mathbf{N}\in \mathcal{N}}\rho(\mathbf{N})\geq~2$.
As a result, $\mathcal{L}(\mathcal{N}) \geq \log \min_{\mathbf{N}\in \mathcal{N}}\rho(\mathbf{N}) >0$.
This yields that $\deg_{G^t}(v)$ grows unboundedly for each node $v\in V(\Gamma)$.
Similar arguments also imply that $\mathcal{L}(\mathcal{M})>0$.

Now, recall that  $V^*(G_s^t) = V(G^s)\setminus V(G^{s-1})$ for all $t \in \mathbb{Z}^{+}$
and let $v^t_{s(t)}$ denote any node of $V^*(G_{s(t)}^t)$ in $G^t$ where $s(t)\leq t$.
By Theorem~\ref{thm:Degree Growth}, 
$\Delta(G^t) \overset{t \to \infty}{\asymp} \exp(\mathcal{L}(\mathcal{N}))^t$ almost surely.
Therefore,
\[
  \hat{\deg}_{G^t}(v_{s(t)}^t)
  \overset{t \to \infty}{\asymp} 
  \frac{\exp(\mathcal{L}(\mathcal{N}))^{t-s(t)}}{\exp(\mathcal{L}(\mathcal{N}))^{t}}
 =\exp(\mathcal{L}(\mathcal{N}))^{-s(t)}
 \in[0,1]\,.
\]
This implies that almost every node $v_{s(t)}^t$ in almost every $\Gamma\in\mathcal{G}$, is stationary.

Finally, we prove that almost every $\Gamma\in\mathcal{G}$ is scale-free.
Let $L\::\:\mathbb{N} \to \mathbb{N}$ be a function satisfying $L(t)/\Delta (G^t) \overset{t \to \infty}{\sim}0$.
Fix $t$ and a random colored substitution network $G^t$; 
when $t$ is large enough, we can find a function $k \::\:\mathbb{N} \to \mathbb{N}$ and $k(t)<t$ such that
\[
  \exp(\mathcal{L}(\mathcal{N}))^{k(t)}\leq L(t) \leq \exp(\mathcal{L}(\mathcal{N}))^{k(t)+1}\,.
\]
By Lemma~\ref{lem:degree range}, almost surely
$c_1^{-1}\exp(\mathcal{L}(\mathcal{N}))^{k(t)}\leq \deg(v_{k(t)}^t)\leq c_1\exp(\mathcal{L}(\mathcal{N}))^{k(t)}$, 
where $c_1$ is a constant depending on $\mathcal{N}$ and $v^t_{k(t)}$.
Take $s_0=\lceil (\log c_1)/(\mathcal{L}(\mathcal{N}))\rceil +1$ so that, for any integer $s>s_0$,
\begin{align*}
  c_1     \exp(\mathcal{L}(\mathcal{N}))^{k(t)-s}&<L(t)\\\text{and}\qquad 
  c_1^{-1}\exp(\mathcal{L}(\mathcal{N}))^{k(t)+s}&>L(t)\,.
\end{align*}
Then $k(t)-s_0<s<k(t)+s_0$ whenever $s\in\mathbb{Z}$ satisfies $\deg(v_{k(t)+s}^t)=L(t)$.
Again, note that Lemma~\ref{lem:degree range} implies that
$c_2^{-1}\exp(\mathcal{L}(\mathcal{M}))^{t-k(t)}\leq\deg(v_{k(t)}^t)\leq c_2 \exp(\mathcal{L}(\mathcal{M}))^{t-k(t)}$.
Therefore, if $P_{L(t)}(G^t)\neq 0$, 
then
\[
P_{L(t)}(G^t)\geq c_2^{-1}\exp(\mathcal{L}(\mathcal{M}))^{t-k(t)-s_0}\,.
\]
Also,\vspace*{-2mm}
\[
  P_{L(t)}(G^t) 
  \leq |\{v_{k(t)+s}^t\::\:-s_0<s<s_0,s\in\mathbb{Z} \}|
  \leq \sum_{i=-s_0}^{s_0}c_2\exp(\mathcal{L}(\mathcal{M}))^{t-k(t)+i}\,.\vspace*{-2mm}
\]
For any sufficiently small $\ell$ with $P_\ell(G^t)>0$, 
we can find large $t$ such that $\ell= L(t)/\Delta(G^t)$.
With $t$ tending to infinity,  $\dim_D (\Gamma)$ equals
\[
  \lim_{\ell \to 0}\frac{\log P_\ell(\Gamma)}{-\log \ell}
 =\lim_{t \to \infty}\frac{\log P_{L(t)}(G^t)}{-\log \frac{L(t)}{\Delta (G^t)}} \,.
\]
As discussed above, 
$P_{L(t)}(G^t) \asymp \exp(\mathcal{L}(\mathcal{M}))^{t-k(t)}$.
This implies that
 \[
 \lim_{t \to \infty}\frac{\log P_{L(t)}(G^t)}{-\log \frac{L(t)}{\Delta (G^t)}} 
 =\lim_{t \to \infty}\frac{\log\exp(\mathcal{L}(\mathcal{M}))^{t-k(t)}}{-\log\frac{\exp(\mathcal{L}(\mathcal{N}))^{k(t)}}
       {\exp(\mathcal{L}(\mathcal{N}))^t}}
 =\frac{\log\exp(\mathcal{L}(\mathcal{M}))}{\log\exp(\mathcal{L}(\mathcal{N}))}\,.
\]
Hence,
\[
\dim_D (\Gamma)
\overset{a.s.}{=\joinrel=}\dfrac{\mathcal{L}(\mathcal{M})}{\mathcal{L}(\mathcal{N})}
\,.\qquad\square
\]


\bibliography{PAMS_RCSN.bib}

\end{document}